\documentclass[11pt,reqno]{amsart}
\usepackage{caption}
 \usepackage{amssymb,amsfonts,amscd,amsbsy, color, amsmath}
\usepackage[mathscr]{eucal}
\usepackage{url}
\usepackage{graphicx}
\usepackage{mathtools}
\usepackage{todonotes}
\usepackage{tabularx}
\usepackage{float, color}
\usepackage{placeins}
\usepackage{tikz}
\usepackage{cite}

\newtheorem{theorem}{Theorem}[section]
\newtheorem{lemma}[theorem]{Lemma}
\newtheorem{proposition}[theorem]{Proposition}
\newtheorem{corollary}[theorem]{Corollary}

\theoremstyle{definition}
\newtheorem{remark}[theorem]{Remark}
\numberwithin{equation}{section}

\makeatletter
\def\imod#1{\allowbreak\mkern5mu({\operator@font mod}\,\,#1)}
\makeatother
\allowdisplaybreaks

\makeatletter
\@namedef{subjclassname@2020}{%
  \textup{2020} Mathematics Subject Classification}
\makeatother

\begin{document}

\title[Rank deviations for overpartitions]{Rank deviations for overpartitions}

\author{Jeremy Lovejoy}
\author{Robert Osburn}

\address{CNRS, Universit{\'e} Paris Cit{\'e}, B{\^a}timent Sophie Germain, case courier 7014, 8 Place Aur{\'e}lie
Nemours, France 75205 Paris Cedex 13, France}
\email{lovejoy@math.cnrs.fr}

\address{School of Mathematics and Statistics, University College Dublin, Belfield, Dublin 4, Ireland}
\email{robert.osburn@ucd.ie}

\subjclass[2020]{11P83, 05A17, 11F37, 11F11, 11F27.}
\keywords{Overpartitions, rank, $M_2$-rank, Appell--Lerch series}

\date{\today}

\begin{abstract}
We prove general fomulas for the deviations of two overpartition ranks from the average. These formulas are in terms of Appell--Lerch series and sums of quotients of theta functions and can be used, among other things, to recover any of the numerous overpartition rank difference identities in the literature. We give two illustrations.
\end{abstract}

\maketitle

\section{Introduction}

A partition of a natural number $n$ is a non-increasing sequence of positive integers whose sum is $n$.   Let $p(n)$ denote the number of partitions of $n$.   The rank of a partition $\lambda$ is the largest part $\ell(\lambda)$ minus the number of parts $n(\lambda)$. Let $N(a,M,n)$ be the function which counts the number of partitions of $n$ into parts with rank congruent to $a$ modulo $M$.   The study of these counting functions is one of the major themes in the theory of partitions.   There is a vast literature on generating functions for $N(a,M,n)$, including identities, inequalities, asymptotics and (mock) modularity properties.

Generating functions for $N(a,M,n)$ were recently revisited by Hickerson and Mortenson \cite{hm2} using a calculus for the Appell--Lerch series 
\begin{equation} \label{al}
m(x,q,z) := \frac{1}{j(z;q)} \sum_{r \in \mathbb{Z}} \frac{(-1)^r q^{\binom{r}{2}} z^r}{1 - q^{r-1} xz}
\end{equation}
that they developed in their work on mock theta functions \cite{hm1}.     Here, $x$, $q$ and $z$ are non-zero complex numbers with $| q | < 1$, neither $z$ nor $xz$ is an integral power of $q$ and
\begin{equation} \label{j}
j(z;q) := (z)_{\infty} (q/z)_{\infty} (q)_{\infty},
\end{equation} 
where
\begin{equation*}
(x)_{\infty} = (x;q)_{\infty} := \prod_{k=0}^{\infty} (1-xq^{k}).
\end{equation*} 
For integers $0 \leq a < M$, Hickerson and Mortenson considered the generating function for the deviation of the rank from the average,
\begin{equation}
D(a,M) := \sum_{n \geq 0} \left(N(a,M,n) - \frac{p(n)}{M}\right)q^n.
\end{equation}
They found simply stated formulas for $D(a,M)$ for any $a$ and $M$ in terms of sums of quotients of theta functions and Appell--Lerch series.    

In this paper we apply the method of Hickerson and Mortenson to overpartition ranks.    An overpartition is a partition in which the first occurrence of each distinct part may be overlined \cite{cl}.   Here there are two ranks of interest.   The first is the same as for ordinary partitions, while the second, called the $M_2$-rank, is  
$$
\text{$M_2$-rank} \hspace{.025in} (\pi) = \bigg \lceil \frac{\ell(\pi)}{2}
\bigg \rceil - n(\pi) + n(\pi_o) - \chi(\pi),
$$
where $\pi_o$ is the subpartition consisting of the odd non-overlined parts and $\chi(\pi) := \chi($the largest part of $\pi$ is odd and non-overlined)\footnote{Throughout, we use the standard notation $\chi(X):=1$ if $X$ is true and $0$ if $X$ is false.} \cite{Lo2}. As with partitions, there is an extensive literature on overpartition ranks.     Generating functions for ranks and rank differences for various small $M$ were computed in \cite{bfhy,cgs,gs,cjs,cjs2,jzz,Lo-Os1,Lo-Os2,wz,hz610}, asymptotic properties were established in \cite{Ci1,Ci2}, and the modularity was investigated in \cite{bl,Dewar,rm2}.

For integers $0 \leq a \leq M$ where $M \geq 2$, we consider the overpartition rank deviations
\begin{equation} \label{dev1}
\overline{D}(a, M) := \sum_{n \geq 0} \Bigl( \overline{N}(a, M, n) - \frac{\overline{p}(n)}{M} \Bigr) q^n
\end{equation}
and
\begin{equation} \label{dev2}
\overline{D}_{2}(a,M) := \sum_{n \geq 0} \Bigl( \overline{N}_{2}(a, M, n) - \frac{\overline{p}(n)}{M} \Bigr) q^n,
\end{equation}
where $\overline{N}(a, M, n)$ denotes the number of overpartitions of $n$ with rank congruent to $a$ modulo $M$,  $\overline{N}_{2}(a, M, n)$ denotes the number of overpartitions of $n$ with $M_2$-rank congruent to $a$ modulo $M$, and $\overline{p}(n)$ is the number of overpartitions of $n$. Note that 
\begin{equation} \label{Dsymmetry}
\overline{D}(a, M)  =  \overline{D}(M-a, M)
\end{equation}
and
\begin{equation} \label{D2symmetry}
\overline{D}_2(a, M)  =  \overline{D}_2(M-a, M),
\end{equation}
which follow from the symmetries \cite{Lo1,Lo2} 
\begin{equation*}
\overline{N}(a,M,n) = \overline{N}(M-a,M,n)
\end{equation*}
and
\begin{equation*}
\overline{N}_2(a,M,n) = \overline{N}_2(M-a,M,n).
\end{equation*}

Our main results are explicit computations for the pairs of deviations $\overline{D}(a,M) + \overline{D}(a-1, M)$ and $\overline{D}_{2}(a, M) + \overline{D}_{2}(a-1, M)$ in terms of Appell--Lerch series and sums of quotients of theta series.    The fact that we consider pairs of devations is due to the form of the relevant generating functions.    It turns out that there is no loss of generality in doing this - see Remarks \ref{remarkModd} and \ref{remarkMeven}.

To state the results, let 

\begin{equation} \label{delta}
\Delta(x,z_1, z_0;q) := \frac{z_0 J_1^3 j(z_1/z_0;q) j(x z_0 z_1; q)}{j(z_0;q) j(z_1;q) j(xz_0;q) j(xz_1;q)}
\end{equation}
and
\begin{equation} \label{Psikndef}
\Psi_{k}^{n}(x,z,z';q) := -\frac{x^k z^{k+1} J_{n^2}^3}{j(z;q) j(z';q^{n^2})} \sum_{t=0}^{n-1} \frac{q^{\binom{t+1}{2} + kt} (-z)^t j(-q^{\binom{n+1}{2} + nk + nt} (-z)^n / z', q^{nt} (xz)^n z'; q^{n^2})}{j(-q^{\binom{n}{2} - nk} (-x)^n z', q^{nt} (xz)^n; q^{n^2})},
\end{equation}
where $J_{m} := (q^m; q^m)_{\infty}$ and $j(z_1,z_2;q) := j(z_1;q)j(z_2;q)$.   Following \cite{hm1}, we use the term ``generic" to mean that the parameters do not cause poles in the Appell--Lerch series or in the quotients of theta functions. 

\begin{theorem} \label{overrankthm}
Let $2 \leq a \leq M$. For generic $z'$, $z'' \in \mathbb{C}$, we have the following generating functions:

\begin{itemize}
\item[($i$)] If $a$ and $M$ are even, then
\begin{equation} \label{c2}
\begin{aligned}
\overline{D}(a, M) + \overline{D}(a-1,M) &= \chi(a=M) + 2(-1)^{\frac{a}{2}} q^{-\frac{a^2}{4}} m((-1)^{\frac{M}{2}+1} q^{\frac{M^2}{4} - \frac{aM}{2}}, q^{\frac{M^2}{2}}, z') \\
& - 2q^{-1} \Psi_{\frac{a}{2}-1}^{\frac{M}{2}}(q^{-1}, -1, z'; q^2) \\
& - \frac{2}{M}  \sum_{j=1}^{M-1}\zeta_{\frac{M}{2}}^{-\frac{aj}{2}} (1 - \zeta_{M}^j) \Delta(\zeta_{M}^{-2j}q, \zeta_{M}^j, -1;q^2). 
\end{aligned}
\end{equation}

\item[($ii$)] If $a$ is even and $M$ is odd, then 
\begin{equation} \label{c6}
\begin{aligned}
\overline{D}(a, M) + \overline{D}(a-1, M) &=  -2 q^{-(\frac{2M-a}{2})^2} (-1)^{\frac{2M-a}{2}} m(q^{M(a-M)}, q^{2M^2}, z') \\
&+ 2q^{-(\frac{M+1-a}{2})^2} (-1)^{\frac{M+1-a}{2}} m(q^{M(a-1)}, q^{2M^2}, z'') \\
&- 2 \Psi_{\frac{2M-a}{2}}^{M}(q,-1,z';q^2)  + 2 \Psi_{\frac{M+1-a}{2}}^{M}(q,-1,z'';q^2) \\
& - \frac{2}{M} \sum_{j=1}^{M-1}  \zeta_{M}^{-aj} (1 - \zeta_{M}^{j}) \Delta(\zeta_{M}^{-2j}q, \zeta_{M}^j, -1;q^2).
\end{aligned}
\end{equation}

\item[($iii$)] If $a$ and $M$ are odd, then
\begin{equation} \label{c7}
\begin{aligned}
\overline{D}(a, M) + \overline{D}(a-1, M) &= \chi(a=M) -2q^{-(\frac{M-a}{2})^2} (-1)^{\frac{M-a}{2}} m(q^{Ma}, q^{2M^2}, z’) \\
& + 2q^{-(\frac{2M-a+1}{2})^2} (-1)^{\frac{2M-a+1}{2}} m(q^{M(a-M-1)}, q^{2M^2}, z'') \\
&- 2 \Psi_{\frac{M-a}{2}}^{M}(q, -1, z’; q^2) + 2 \Psi_{\frac{2M-a+1}{2}}^{M}(q,-1,z'';q^2) \\
& -  \frac{2}{M} \sum_{j=1}^{M-1} \zeta_{M}^{-aj} (1 - \zeta_{M}^{j}) \Delta(\zeta_{M}^{-2j}q, \zeta_{M}^j, -1;q^2).
\end{aligned}
\end{equation}
\end{itemize}
\end{theorem}

\begin{theorem}  \label{M2overrankthm}
Let $1 \leq a \leq M-1$. For generic $z'$, $z'' \in \mathbb{C}$, we have the generating function
\begin{equation} \label{newc5}
\begin{aligned}
\overline{D}_{2}(a, M) + \overline{D}_{2}(a-1, M) &= \chi(a=1) + 2(-1)^a q^{-a^2} m((-1)^{M+1} q^{M^2 - 2Ma}, q^{2M^2}, z') \\
&+ 2(-1)^a q^{-a^2+2a-1} m((-1)^{M+1} q^{M^2 - 2M(a-1)}, q^{2M^2}, z'')  \\
&+ 2 \Psi_{a}^{M}(q,-1,z';q^2)  - 2 \Psi_{a-1}^{M}(q,-1,z'';q^2) \\
&+ \frac{2}{M} \sum_{j=1}^{M-1} \zeta_{M}^{-aj} (1 - \zeta_{M}^j) \Delta(\zeta_{M}^{j} q, q, -1; q^2).
\end{aligned}
\end{equation}
\end{theorem}

\begin{remark}
If $a$ is odd and $M$ is even, then $\overline{D}(a, M) + \overline{D}(a-1,M)$ is computed from Theorem \ref{overrankthm} (i) using the fact that
\begin{equation*}
\overline{D}(a, M) + \overline{D}(a-1,M) =  \overline{D}(M-a+1, M) + \overline{D}(M-a, M),
\end{equation*}
which follows from \eqref{Dsymmetry}.
\end{remark}

\begin{remark} \label{remarkModd}
Note that for $M$ odd we have 
\begin{equation*}
\overline{D} \Bigl(\frac{M+1}{2}, M \Bigr) + \overline{D} \Bigl(\frac{M-1}{2}, M \Bigr) =  2\overline{D} \Bigl(\frac{M-1}{2}, M \Bigr), 
\end{equation*}
and so Theorem \ref{overrankthm} can be used to find a formula for any single $\overline{D}(a, M)$.   Thus, there is no loss of generality in considering the pairwise sums of the rank deviations for $M$ odd.  A similar remark applies in the case of the $M_2$-rank in Theorem \ref{M2overrankthm}.     
\end{remark}

\begin{remark} \label{remarkMeven}
When $M$ is even, there is also no loss in generality in considering the sums of the rank deviations if one combines Theorems \ref{overrankthm} and \ref{M2overrankthm} with the formulas \cite{Br-Lo1}
\begin{equation} \label{Tdef}
\overline{D}(0,2) = 2\frac{(-q)_{\infty}}{(q)_{\infty}}\sum_{n \in \mathbb{Z}} \frac{(-1)^nq^{n^2+n}}{(1+q^n)^2}
\end{equation}
and
\begin{equation}
\overline{D}_2(0,2) = 2\frac{(-q)_{\infty}}{(q)_{\infty}}\sum_{n \in \mathbb{Z}} \frac{(-1)^nq^{n^2+2n}}{(1+q^{2n})^2}.
\end{equation}
For an example, see Corollary \ref{firstcor}.
\end{remark}

\begin{remark}
The results in Theorem \ref{overrankthm} may be compared with similar formulas of Zhang, where only one rank deviation is involved in each case \cite{hz}. Her work is based on an erroneous identity, however, casting some doubt on the veracity of her results (see \cite[page 691]{hz}, quoted from \cite[page 251]{cjs}).     
\end{remark}

For small $M$ the sums of quotients of theta functions in Theorems \ref{overrankthm} and \ref{M2overrankthm} simplify nicely for certain choices of the free parameters $z'$ and $z''$.   We illustrate this in Section 4 for the case of the ordinary overpartition rank when $M=3$ and $6$, and leave other examples to the reader.    In theory, Theorems \ref{overrankthm} and \ref{M2overrankthm} can be used to reduce the proof of any observed generating function for $\overline{D}(a,M)$ or $\overline{D}_2(a,M)$ to a verification of an identity involving modular forms, although in practice, as $M$ grows, this involves considerable bookkeeping and extensive computations involving modular forms on $\Gamma_1(N)$.   The computations can be simplified using a dedicated computer package such as the one described by Frye and Garvan \cite{FG}.      

While our focus here is on identities, we note that Theorems \ref{overrankthm} and \ref{M2overrankthm} can be used to establish the modularity of rank generating functions in arithmetic progressions.   Note for example that if we take $z' = z'' = -1$ then the non-modular part in equations \eqref{c6}--\eqref{newc5} is supported on only two arithmetic progressions modulo $M$.   In a manner similar to how the work of Hickerson and Mortenson yields stronger versions of \cite[Theorems 1.3 and 1.4]{bo} and \cite[Theorem 1.1]{bor}, our results could be used to give statements like those in \cite[Theorems 1.1 and 1.6]{bl}, \cite[Theorem 1.1]{Dewar} and \cite[Theorems 1.1 and 1.2]{rm2}.   For more details on how this works, see \cite[Section 7]{hm2}.

The paper is organized as follows. In Section 2, we discuss the relevant background on properties of Appell--Lerch series and prove two key formulas for computing rank deviations. In Section 3, we prove Theorems \ref{overrankthm} and \ref{M2overrankthm}. In Section 4, we illustrate our results with two examples (see Propositions \ref{3example} and \ref{6example}) and show how one can recover certain rank difference generating functions studied in \cite{jzz, Lo-Os1}. 

\section{Preliminaries}

We begin by recalling the following two results on Appell--Lerch series which will play an important role in our calculations. The first relates two such series with different generic parameters $z_1$ and $z_0$ \cite[Theorem 3.3]{hm1} while the second is an orthogonality result \cite[Theorem 3.9]{hm1}.

\begin{lemma} \label{switch} For generic $x$, $z_0$ and $z_1 \in \mathbb{C}^{*}$
\begin{equation}
m(x,q,z_1) - m(x,q,z_0) = \Delta(x,z_1, z_0;q),
\end{equation}
where $ \Delta(x,z_1, z_0;q)$ is given by (\ref{delta}).
\end{lemma}

\begin{lemma} \label{orthog} Let $n$ and $k$ be integers with $0 \leq k < n$. Let $\omega$ be a primitive $n$-th root of unity. Then
\begin{equation} 
\sum_{t=0}^{n-1} \omega^{-kt} m(\omega^t x, q, z) = n q^{-\binom{k+1}{2}} (-x)^k m(-q^{\binom{n}{2} - nk} (-x)^n, q^{n^2}, z') + n \Psi_{k}^{n}(x,z,z';q),
\end{equation}
where $\Psi_{k}^{n}(x,z,z';q)$ is given by (\ref{Psikndef}).
\end{lemma}

We also require \cite[Eq. (3.2b))]{hm1}
\begin{equation} \label{flip}
m(x,q,z) = x^{-1} m(x^{-1}, q, z^{-1})
\end{equation}
and the following well-known fact. Let $\zeta_n$ be an $n$-th root of unity and $s \in \mathbb{Z}$. Then

\begin{equation} \label{sim}
\sum_{j=0}^{n-1} \zeta_n^{sj}= \begin{cases}
 n & \text{if $s \equiv 0 \pmod{n}$}, \\
 0 & \text{otherwise.}
\end{cases}
\end{equation}
We now prove two key formulas which allow one to explicitly compute rank deviations. Let $\overline{N}(m, n)$ denote the number of overpartitions of $n$ with rank equal to $m$,

\begin{equation} \label{gen1}
\overline{R}(z;q) := \sum_{\substack{m \in \mathbb{Z} \\ n \geq 0}} \overline{N}(m, n) z^m q^n
\end{equation}
and
\begin{equation} \label{sdef}
\overline{S}(z;q) := (1+z) \overline{R}(z;q).
\end{equation}
Similarly, let $\overline{N}_{2}(m, n)$ denote the number of overpartitions of $n$ with $M_2$-rank equal to $m$,
\begin{equation} \label{gen2}
\overline{R}_{2}(z;q) := \sum_{\substack{m \in \mathbb{Z} \\ n \geq 0}} \overline{N}_{2}(m, n) z^m q^n
\end{equation}
and
\begin{equation} \label{s2def}
\overline{S}_{2}(z;q) := (1+z) \overline{R}_{2}(z;q).
\end{equation}

\begin{proposition} We have
\begin{equation} \label{overkey}
\frac{1}{M} \sum_{j=1}^{M-1} \zeta_{M}^{-aj} \overline{S}(\zeta_M^{j}; q) = \overline{D}(a,M) + \overline{D}(a-1,M)
\end{equation}
and
\begin{equation} \label{overkey2}
\frac{1}{M} \sum_{j=1}^{M-1} \zeta_{M}^{-aj} \overline{S}_{2}(\zeta_M^{j}; q) = \overline{D}_{2}(a, M) + \overline{D}_{2}(a-1,M).
\end{equation}
\end{proposition}

\begin{proof}
First note that from (\ref{gen1}) and (\ref{sdef}), it follows that
\begin{equation*}
\overline{S}(z;q) = \sum_{\substack{m \in \mathbb{Z} \\ n \geq 0}} \Bigl( \overline{N}(m, n) + \overline{N}(m-1, n) \Bigr) z^m q^n
\end{equation*}
and so applying (\ref{sim}) yields
\begin{align*}
\frac{1}{M} \sum_{j=1}^{M-1} \zeta_{M}^{-aj} \overline{S}(\zeta_M^{j}; q) & = \frac{1}{M} \sum_{j=0}^{M-1} \zeta_{M}^{-aj} \overline{S}(\zeta_M^{j}; q) - \frac{2}{M} \sum_{n \geq 0} \overline{p}(n) q^n \\
& = \sum_{n \geq 0} \Biggl( \sum_{k \in \mathbb{Z}} \overline{N}(kM + a, n) + \overline{N}(kM + a-1, n) \Biggr) q^n - \frac{2}{M} \sum_{n \geq 0} \overline{p}(n) q^n \\
& =  \overline{D}(a,M) + \overline{D}(a-1,M).
\end{align*}
Using a similar argument, (\ref{overkey2}) is obtained using (\ref{sim}), (\ref{gen2}) and (\ref{s2def}). 
\end{proof}

\section{Proofs of Theorems \ref{overrankthm} and \ref{M2overrankthm}}

\begin{proof}[Proof of Theorem \ref{overrankthm}]
The overall strategy in proving (\ref{c2})--(\ref{c7}) is to decompose the left-hand side of (\ref{overkey}) into sums which are amenable to simplification via (\ref{orthog}) and (\ref{sim}). First, we combine \cite[Eq. (1.1)]{Lo1},  \cite[Eq. (2.10)]{Mc1} and \cite[Eq. (4.7)]{hm1} to obtain
\begin{equation} \label{stom}
\overline{S}(z;q)  = (1-z) (1 - 2m(z^{-2}q, q^2, z)).
\end{equation}
We then apply (\ref{switch}) to switch the third parameter $z$ in the Appell--Lerch series in (\ref{stom}) to $-1$ in order to avoid poles. This yields
\begin{align} \label{genform}
\displaystyle \frac{1}{M} \sum_{j=1}^{M-1} \zeta_{M}^{-aj} \overline{S}(\zeta_M^{j}; q) & = \displaystyle \frac{1}{M} \sum_{j=1}^{M-1} \zeta_{M}^{-aj} (1 - \zeta_M^j)(1 - 2 m(\zeta_M^{-2j} q, q^2, \zeta_M^j) )\nonumber \\
& = \frac{1}{M} \Biggl( \sum_{j=0}^{M-1} \zeta_{M}^{-aj} - \sum_{j=0}^{M-1} \zeta_{M}^{-(a-1)j} \Biggr) - \frac{2}{M} \Biggl ( \sum_{j=0}^{M-1} \zeta_{M}^{-aj} m(\zeta_{M}^{-2j}q, q^2, -1) \nonumber \\
& - \sum_{j=0}^{M-1} \zeta_{M}^{-(a-1)j} m(\zeta_{M}^{-2j}q, q^2, -1) + \sum_{j=1}^{M-1} \zeta_{M}^{-aj} (1 - \zeta_M^j) \Delta(\zeta_{M}^{-2j}q, \zeta_{M}^j, -1; q^2) \Biggr).
\end{align}

To prove (\ref{c2}), we first use (\ref{sim}) to observe that the first two sums in the second line of (\ref{genform}) equal $0$ if $a<M$ and $1$ if $a=M$. For the third sum, we split it into two further sums. We then reindex the resulting second sum by $j \to \frac{M}{2} + j$, use that $\zeta_{M}^{-2} = \zeta_{\frac{M}{2}}^{-1}$, write $a=2t$ where $t>0$, apply (\ref{flip}) and take $k=t-1$, $n=\frac{M}{2}$, $z=-1$, $x=q^{-1}$ and $q \to q^2$ in (\ref{orthog}) to obtain
\begin{align*}
\sum_{j=0}^{M-1} \zeta_{M}^{-aj} m(\zeta_{M}^{-2j} q, q^2, -1) & = \sum_{j=0}^{\frac{M}{2} - 1} \zeta_{M}^{-aj} m(\zeta_{M}^{-2j}q, q^2, -1) + \sum_{j=\frac{M}{2}}^{M-1} \zeta_{M}^{-aj} m(\zeta_{M}^{-2j}q, q^2, -1) \\
& = 2 \sum_{j=0}^{\frac{M}{2} - 1} \zeta_{M}^{-aj} m(\zeta_{\frac{M}{2}}^{-j} q, q^2, -1) \\
& = M (-1)^{t-1} q^{-t^2} m((-1)^{\frac{M}{2} + 1} q^{\frac{M^2}{4} - Mt}, q^{\frac{M^2}{2}}, z') \\
& \qquad \qquad \qquad \qquad \qquad \qquad \qquad \qquad + q^{-1} M \Psi_{t-1}^{\frac{M}{2}}(q^{-1}, -1, z'; q^2).
\end{align*}
For the fourth sum, we similarly split it into two further sums, then reindex the resulting second sum by $j \to \frac{M}{2} + j$. We then use that $\zeta_{M}^{\frac{M}{2}} = -1$ to obtain $0$. In total, this yields (\ref{c2}).

To prove (\ref{c6}), we first use (\ref{sim}) to note that the first two sums in the second line of (\ref{genform}) equal $0$. For the third sum, we first use that if $\zeta_M$ is a primitive $M$-th root of unity, then so is $\zeta_M^2$ as $M$ is odd and then take $k=\frac{2M-a}{2}$, $n=M$, $z=-1$, $x=q$ and $q \to q^2$ in (\ref{orthog}) to obtain
\begin{align*}
\sum_{j=0}^{M-1} \zeta_{M}^{-aj} m(\zeta_{M}^{-2j} q, q^2, -1) & = M q^{-(\frac{2M-a}{2})^2} (-1)^{\frac{2M-a}{2}} m(q^{M(a-M)}, q^{2M^2}, z') \\
& \qquad \qquad \qquad \qquad \qquad \qquad \qquad \qquad + M \Psi_{\frac{2M-a}{2}}^{M}(q, -1, z' ;q^2).
\end{align*}
For the fourth sum, we take $k=\frac{M+1-a}{2}$, $n=M$, $z=-1$, $x=q$ and $q \to q^2$ in (\ref{orthog}) to obtain
\begin{align*}
\sum_{j=0}^{M-1} \zeta_{M}^{-(a-1)j} m(\zeta_{M}^{-2j} q, q^2, -1) & = M q^{-(\frac{a+M-1}{2})^2} (-1)^{\frac{a+M-1}{2}} m(q^{M-aM}, q^{2M^2}, z'') \\
& \qquad \qquad \qquad \qquad \qquad \qquad \qquad \qquad  + M \Psi_{\frac{a+M-1}{2}}^{M}(q,-1,z''; q^2).
\end{align*}
In total, this yields (\ref{c6}).

Finally, to prove (\ref{c7}), we first use (\ref{sim}) to see that the first two sums in the second line of (\ref{genform}) equal $0$ unless $a=M$, in which case the sum is $M$.  For the third sum, we take $k=\frac{M-a}{2}$, $n=M$, $z=-1$, $x=q$ and $q \to q^2$ in (\ref{orthog}) to obtain
\begin{align*}
\sum_{j=0}^{M-1} \zeta_{M}^{-aj} m(\zeta_{M}^{-2j} q, q^2, -1) & = M q^{-(\frac{M-a}{2})^2} (-1)^{\frac{M-a}{2}} m(q^{aM}, q^{2M^2}, z') \\
& \qquad \qquad \qquad \qquad \qquad \qquad \qquad \qquad + M \Psi_{\frac{M-a}{2}}^{M}(q, -1, z'; q^2). 
\end{align*}
For the fourth sum, we take $k=\frac{2M-a+1}{2}$, $n=M$, $z=-1$, $x=q$ and $q \to q^2$ in (\ref{orthog}) to obtain
\begin{align*}
\sum_{j=0}^{M-1} \zeta_{M}^{-(a-1)j} m(\zeta_{M}^{-2j} q, q^2, -1) & = M q^{-(\frac{2M-a+1}{2})^2} (-1)^{\frac{2M-a+1}{2}} m(q^{M(a-M-1)}, q^{2M^2}, z'') \\
& \qquad \qquad \qquad \qquad \qquad \qquad \qquad \qquad + M \Psi_{\frac{2M-a+1}{2}}^{M}(q,-1,z''; q^2).
\end{align*}
In total, this yields (\ref{c7}).
\end{proof}

\begin{proof}[Proof of Theorem \ref{M2overrankthm}]
The overall strategy in proving (\ref{newc5}) is to decompose the left-hand side of (\ref{overkey2}) into sums in which (\ref{orthog}) and (\ref{sim}) are applicable. To do this, we first combine \cite[Eq. (2.2)]{rm} and \cite[Eqs. (3.2a), (3.2b), (3.2e)]{hm1} to obtain
\begin{equation} \label{stom2}
\overline{S}_{2}(z;q) = -(1-z) + 2(1-z) m(zq,q^2,q).
\end{equation}
We then apply (\ref{switch}) to switch the third parameter $q$ in the Appell--Lerch series in (\ref{stom2}) to $-1$ in order to avoid poles. This yields
\begin{align} \label{genform2}
\displaystyle \frac{1}{M} \sum_{j=1}^{M-1} \zeta_{M}^{-aj} \overline{S}_{2}(\zeta_M^{j}; q) & = \displaystyle -\frac{1}{M} \sum_{j=0}^{M-1} \zeta_{M}^{-aj} (1 - \zeta_M^j) + \frac{2}{M} \sum_{j=0}^{M-1} \zeta_{M}^{-aj} (1 - \zeta_M^j) m(\zeta_M^j q, q^2, q) \nonumber \\
& = -\frac{1}{M} \Biggl( \sum_{j=0}^{M-1} \zeta_{M}^{-aj} - \sum_{j=0}^{M-1} \zeta_{M}^{-(a-1)j} \Biggr) + \frac{2}{M} \Biggl ( \sum_{j=0}^{M-1} \zeta_{M}^{-aj} m(\zeta_{M}^{j}q, q^2, -1) \nonumber \\
& - \sum_{j=0}^{M-1} \zeta_{M}^{-(a-1)j} m(\zeta_{M}^{j}q, q^2, -1) + \sum_{j=1}^{M-1} \zeta_{M}^{-aj} (1 - \zeta_M^j) \Delta(\zeta_{M}^{j} q, q, -1; q^2) \Biggr).
\end{align}

To prove (\ref{newc5}), we first use (\ref{sim}) to observe that the first two sums in (\ref{genform2}) equal $0$ unless $a=1$, in which case we obtain $1$. For the third sum, we take $k=a$, $n=M$, $z=-1$, $x=q$ and $q \to q^2$ in (\ref{orthog}) to obtain
\begin{align*}
\sum_{j=0}^{M-1} \zeta_{M}^{-aj} m(\zeta_{M}^{j}q, q^2, -1) & = M q^{-a^2}  (-1)^{a} m((-1)^{M+1} q^{M^2 - 2Ma}, q^{2M^2}, z') + M \Psi_{a}^{M}(q, -1, z'; q^2).
\end{align*}
For the fourth sum, we take $k=a-1$, $n=M$, $z=-1$, $x=q$ and $q \to q^2$ in (\ref{orthog}) to obtain
\begin{align*}
\sum_{j=0}^{M-1} \zeta_{M}^{-(a-1)j} m(\zeta_{M}^{j}q, q^2, -1) & = M q^{-a^2 + 2a-1} (-1)^{a-1} m((-1)^{M+1} q^{M^2 - 2M(a-1)}, q^{2M^2}, z'') \\
& \qquad \qquad \qquad \qquad \qquad \qquad \qquad \qquad \qquad + M \Psi_{a-1}^{M}(q, -1, z''; q^2).
\end{align*}
In total, this yields (\ref{newc5}).
\end{proof}

\section{Rank differences revisited}
In this section, we demonstrate the universality of our results for computing overpartition rank deviation generating functions by exhibiting the cases $M=3$ and $6$ of Theorem \ref{overrankthm}.   As corollaries we show how to recover known generating functions for rank differences for overpartitions \cite{jzz, Lo-Os1}.    

We will make use of the function
\begin{equation} \label{hdef} 
h(x;q)  := \frac{(-q)_{\infty}}{(q)_{\infty}}\sum_{n \in \mathbb{Z}} \frac{(-1)^nq^{n^2+n}}{1-xq^n},
\end{equation}
which appears frequently in generating functions for overpartition ranks in the literature and is related to the Appell--Lerch series via \cite[Eq. (4.7)]{hm1}
\begin{equation} \label{h=m}
h(x;q) = -x^{-1}m(x^{-2}q,q^2,x).
\end{equation}
We also need the relations \cite[Eq. (3.1)]{Mc1}
\begin{equation} \label{hx+h-x}
h(x;q) + h(-x;q) = 2 \frac{J_2^4}{J_1^2j(x^2;q^2)}
\end{equation} 
and
\begin{equation} \label{hx=hq/x}
h(x;q) =h(x^{-1}q;q).
\end{equation}
The latter follows upon replacing $n$ by $-n-1$ in \eqref{hdef}.  

\begin{proposition} \label{3example}
We have
\begin{align}
\overline{D}(3,3) + \overline{D}(2,3) &= -2q^2h(q^6;q^9) + \frac{1}{3}\frac{J_2J_3^6J_{18}}{J_1^2J_6^3J_9^2}, \label{3exampleeq1} \\
\overline{D}(2,3) + \overline{D}(1,3) = 2\overline{D}(2,3) &= 4q^2h(q^6;q^9) - \frac{2}{3}\frac{J_2J_3^6J_{18}}{J_1^2J_6^3J_9^2}. \label{3exampleeq2}
\end{align}
\end{proposition}

\begin{proof}
We give details for \eqref{3exampleeq1}.  Equation \eqref{3exampleeq2} follows from \eqref{3exampleeq1} together with the fact that for any $n$ one has
\begin{equation*}
\overline{N}(3,3,n) + 2\overline{N}(2,3,n) = \overline{p}(n),
\end{equation*}
which gives
\begin{equation*}
\overline{D}(3,3) + 2\overline{D}(2,3) = 0.
\end{equation*}
To begin, the case $a=M=3$ of \eqref{c7} with $z'=-1$ and $z'' = q^6$ gives 
\begin{equation} \label{ex1start}
\begin{aligned}
\overline{D}(3,3) + \overline{D}(2,3) &= 1 - 2m(q^9,q^{18},-1) + 2q^{-4}m(q^{-3},q^{18},q^6) \\
&- 2\Psi_{0}^3(q,-1,-1;q^2) + 2\Psi_2^3(q,-1,q^6;q^2) \\
&- \frac{2}{3} \sum_{j=1}^2 (1-\zeta_3^j)\Delta(\zeta_3^{-2j}q,\zeta_3^j,-1;q^2).
\end{aligned}
\end{equation}
Using \eqref{h=m} 
and the identity \cite[Eq. (3.3)]{hm1}
\begin{equation} \label{mhalf}
m(q, q^2, -1) = \frac{1}{2},
\end{equation}
the first line on the right-hand side of \eqref{ex1start} is
\begin{equation} \label{ex1eq2}
1 - 2m(q^9,q^{18},-1) + 2q^{-4}m(q^{-3},q^{18},q^6) = -2q^2h(q^6;q^9).
\end{equation}
As for the second line, we expand the first term using \eqref{Psikndef} to obtain
\begin{equation} \label{Psi03}
\begin{aligned}
\Psi_{0}^{3}(q,-1,-1;q^2) &= \frac{J_{18}^3}{j(-1;q^2)j(-1;q^{18})}\Bigg(\frac{j(q^{12};q^{18})j(q^3;q^{18})}{j(-q^9;q^{18})j(-q^3;q^{18})}  \\
&+ q^2\frac{j(q^{18};q^{18})j(q^9;q^{18})}{j(-q^{9}; q^{18})j(-q^9; q^{18})} + q^6\frac{j(q^{24}; q^{18})j(q^{15}; q^{18})}{j(-q^9; q^{18})j(-q^{15}; q^{18})} \Bigg).
\end{aligned}
\end{equation} 
From (\ref{j}), we have the following properties
\begin{align}
j(x;q) &= j(x^{-1}q;q), \label{j1} \\
j(q^nx;q) &= (-1)^nq^{-\binom{n}{2}}x^{-n}j(x;q) \label{j2}.
\end{align}
Using (\ref{j1}) and (\ref{j2}) we find that inside the parentheses on the right-hand side of \eqref{Psi03}, the middle term vanishes and the other two terms cancel.    This gives
\begin{equation} \label{ex1eq3}
\Psi_0^3(q,-1,-1;q^2) = 0.
\end{equation}
Next, we have from \eqref{Psikndef}
\begin{equation} \label{ex1eq4}
\begin{aligned}
2\Psi_2^3(q,-1,q^6;q^2) &= 2\frac{q^2J_{18}^3}{j(-1;q^2)j(q^6; q^{18})}\Bigg(\frac{j(-q^{18}; q^{18})j(-q^9; q^{18})}{j(q^3; q^{18})j(-q^3; q^{18})}  \\
&+ \frac{j(-q^{6}; q^{18})j(-q^{15}; q^{18})}{j(q^{3}; q^{18})j(-q^9; q^{18})} + q^{-1}\frac{j(-q^{12}; q^{18})j(-q^{3}; q^{18})}{j(q^3; q^{18})j(-q^{15}; q^{18})} \Bigg),
\end{aligned}
\end{equation}
which unfortunately does not appear to simplify in any significant way.   

Finally, we treat the sum in the last line of \eqref{ex1start}.    Using properties of $\zeta_3$ and a short computation we have
\begin{align*}
(1-\zeta_3) \Delta(\zeta_3^{-2}q,\zeta_3,-1;q^2)  
&= -\frac{1}{2}(1+\zeta_3) \frac{J_1J_2^4J_{12}}{J_3J_4^3J_6}, \\
(1-\zeta_3^2) \Delta(\zeta_3^{-4}q,\zeta_3^2,-1;q^2) &= -\frac{1}{2}(1+\zeta_3^{-1}) \frac{J_1J_2^4J_{12}}{J_3J_4^3J_6},
\end{align*}
and so
\begin{equation} \label{ex1eq5}
-\frac{2}{3}\sum_{j=1}^2 \zeta_3^{-j}(1-\zeta_3^j)\Delta(\zeta_3^{-2j}q,\zeta_3^j,-1;q^2) = \frac{1}{3}\frac{J_1J_2^4J_{12}}{J_3J_4^3J_6}.
\end{equation}
Comparing equations \eqref{ex1start}, \eqref{ex1eq2} and \eqref{ex1eq3}--\eqref{ex1eq5}, we are left to prove that
\begin{equation} \label{toprove}
\eqref{ex1eq4} + \eqref{ex1eq5} = \frac{1}{3}\frac{J_2J_3^6J_{18}}{J_1^2J_6^3J_9^2}.
\end{equation}
This is an identity between modular forms, which can be verified with the following finite computation.    
Using the dictionary
\begin{equation} \label{dic}
\begin{aligned} 
j(-1;q) &= 2 \frac{J_2^2}{J_1}, \quad j(q;q^2) = \frac{J_1^2}{J_2}, \quad j(-q;q^2) = \frac{J_2^5}{J_1^2 J_4^2}, \quad j(q;q^3) = J_1, \\
j(-q;q^3) &= \frac{J_2 J_3^2}{J_1 J_6}, \quad j(q;q^6) =  \frac{J_1 J_6^2}{J_2 J_3}, \quad j(-q;q^6) = \frac{J_2^2 J_3 J_{12}}{J_1 J_4 J_6},
\end{aligned}
\end{equation}
identity \eqref{toprove} may be written as
\begin{equation} \label{toprove2}
2q^2 \frac{J_2J_{12}J_{18}^6}{J_4^2J_6^2J_9^2J_{36}} + q^2 \frac{J_2J_6J_9^4J_{36}^2}{J_3^2J_4^2J_{18}^3} + q\frac{J_2J_9J_{12}J_{18}^3}{J_3J_4^2J_6J_{36}} +\frac{1}{3}\frac{J_1J_2^4J_{12}}{J_3J_4^3J_6} =  \frac{1}{3}\frac{J_2J_3^6J_{18}}{J_1^2J_6^3J_9^2}.
\end{equation}
Multiplying both sides by $J_1^8J_6^7J_9^2/J_2J_{18}$ and converting to the notation of modular forms using the Dedekind eta-function, $\eta(q) := q^{1/24} J_1$, equation \eqref{toprove2} is equivalent to
\begin{equation} \label{toprove3}
\begin{aligned}
2\frac{\eta^8(q)\eta^5(q^6)\eta(q^{12})\eta^5(q^{18})}{\eta^2(q^4)\eta(q^{36})} &+ \frac{\eta^8(q)\eta^8(q^6)\eta^6(q^9)\eta^2(q^{36})}{\eta^2(q^3)\eta^2(q^4)\eta^4(q^{18})} + \frac{\eta^8(q)\eta^6(q^6)\eta^3(q^9)\eta(q^{12})\eta^2(q^{18})}{\eta(q^3)\eta^2(q^4)\eta(q^{36})} \\
&+ \frac{1}{3} \frac{\eta^9(q)\eta^3(q^2)\eta^6(q^6)\eta^2(q^9)\eta(q^{12})}{\eta(q^3)\eta^3(q^4)\eta(q^{18})} = \frac{1}{3}\eta^6(q)\eta^6(q^3)\eta^4(q^6).
\end{aligned}
\end{equation}
Now, using standard criteria for the modularity of quotients of eta-functions \cite[Theorems 1.64 and 1.65]{OnoCBMS}, one may check that each of the five quotients above is a holomorphic modular form of weight $8$ and level $36$.    Therefore, using well-known bounds (e.g., \cite[Theorem 2.58]{OnoCBMS}), to establish \eqref{toprove3} we need only verify the $q$-expansions of both sides agree up to $q^{48}$.     This completes the proof of \eqref{3exampleeq1}.
\end{proof}

Proposition \ref{3example} can be used to recover the following overpartition rank difference generating functions of the authors \cite[Theorem 1.1]{Lo-Os1}.

\begin{corollary} \label{3corollary}
We have
\begin{align}
\sum_{n \geq 0} \left(\overline{N}(0,3,3n) - \overline{N}(1,3,3n)\right)q^n &= \frac{J_3^4J_2}{J_1^2J_6^2}, \label{3cor1}\\
\sum_{n \geq 0} \left(\overline{N}(0,3,3n+1) - \overline{N}(1,3,3n+1)\right)q^n &= 2\frac{J_3J_6}{J_1}, \label{3cor2} \\
\sum_{n \geq 0} \left(\overline{N}(0,3,3n+2) - \overline{N}(1,3,3n+2)\right)q^n &= 4 \frac{J_6^4}{J_2J_3^2} - 6h(q;q^3). \label{3cor3}
\end{align}
\end{corollary}

\begin{remark}
We note that \cite[Eq. (2)]{Lo-Os1} has an extra $-1$ due to a difference of convention.   Here we have assumed that the empty overpartition of $0$ has rank $0$, and in \cite{Lo-Os1} it was assumed that the rank of the empty overpartition is undefined.    
\end{remark}

\begin{proof}[Proof of Corollary \ref{3corollary}]
From \cite[Proposition 11]{FJM} we have
\begin{equation} 
\frac{J_2}{J_1^2} = \frac{J_6^4J_9^6}{J_3^{8}J_{18}^3} + 2q \frac{J_6^3J_9^3}{J_3^7} + 4q^2\frac{J_6^2J_{18}^3}{J_3^6},
\end{equation}
which gives
\begin{equation} \label{3dissection}
\frac{J_2J_3^6J_{18}}{J_1^2J_6^3J_9^2} = \frac{J_6J_9^4}{J_3^2J_{18}^2} + 2q \frac{J_9J_{18}}{J_3} + 4q^2 \frac{J_{18}^4}{J_6J_9^2}.
\end{equation}
Using (\ref{3exampleeq1}), (\ref{3exampleeq2}) and (\ref{3dissection}) we have
\begin{equation} \label{readoff}
\begin{aligned}
\sum_{n \geq 0} \left(\overline{N}(0,3,n) - \overline{N}(1,3,n)\right)q^n &= \overline{D}(3,3) + \overline{D}(2,3) - 2\overline{D}(2,3) \\
&= -6q^2h(q^6;q^9) + \frac{J_6J_9^4}{J_3^2J_{18}^2} + 2q \frac{J_9J_{18}}{J_3} + 4q^2 \frac{J_{18}^4}{J_6J_9^2}.
\end{aligned}
\end{equation}
Equations (\ref{3cor1})--(\ref{3cor3}) can now be read off using \eqref{hx=hq/x} and \eqref{readoff}.
\end{proof}

Next, we turn to the case $M=6$.    
\begin{proposition} \label{6example}
We have
\begin{align}
\overline{D}(0,6) + \overline{D}(1,6) &= \frac{2}{3}\frac{J_2^4}{J_1^2J_6}, \label{6exampleeq1} \\
\overline{D}(1,6) + \overline{D}(2,6) &= 2q^2h(q^6;q^9) - \frac{1}{3}\frac{J_2J_3^6J_{18}}{J_1^2J_6^3J_9^2}, \label{6exampleeq2} \\
\overline{D}(2,6) + \overline{D}(3,6) &= -2q^2h(q^6;q^9) + \frac{1}{3}\frac{J_2J_3^6J_{18}}{J_1^2J_6^3J_9^2} - \frac{2}{3}\frac{J_2^4}{J_1^2J_6} . \label{6exampleeq3}
\end{align}
\end{proposition}

\begin{proof}
We need only show \eqref{6exampleeq1}.   Equation \eqref{6exampleeq2} follows from \eqref{3exampleeq2} and the fact that 
\begin{equation*}
\overline{D}(1,6) + \overline{D}(2,6) =  \overline{D}(1,6) + \overline{D}(4,6) = \overline{D}(1,3),
\end{equation*}
while equation \eqref{6exampleeq3} follows from \eqref{6exampleeq1} and \eqref{6exampleeq2} together with 
\begin{equation*}
\overline{D}(0,6) + \overline{D}(1,6) + \overline{D}(1,6) + \overline{D}(2,6) = \overline{D}(0,6) + \overline{D}(1,6) + \overline{D}(4,6) + \overline{D}(5,6)
\end{equation*}
and
\begin{equation*}
\sum_{i=0}^{5} \overline{D}(i,6) = 0.
\end{equation*}

To begin we take $a=6$ and $M=6$ in \eqref{c2} with $z' = -1$ and apply \eqref{mhalf} to find that 
\begin{equation} \label{d6}
\begin{aligned}
\overline{D}(0,6) + \overline{D}(1,6) = \overline{D}(6,6) + \overline{D}(5,6) &= -2q^{-1}\Psi_2^3(q^{-1},-1,-1;q^2)  \\
&- \frac{1}{3}\sum_{j=1}^{5} (1 - \zeta_{6}^j) \Delta(\zeta_{6}^{-2j}q, \zeta_{6}^j, -1;q^2).
\end{aligned}
\end{equation}
A short computation as in the proof of Proposition \ref{3example} gives 
\begin{equation*}
\Psi_2^3(q^{-1},-1,-1;q^2) = 0,
\end{equation*}
leaving us to evaluate
\begin{equation*}
-\frac{1}{3}\sum_{j=1}^{5} (1 - \zeta_{6}^j) \Delta(\zeta_{6}^{-2j}q, \zeta_{6}^j, -1;q^2).
\end{equation*}
To this end, we use \eqref{delta} and properties of $\zeta_6$ to calculate that
\begin{align*}
(1 - \zeta_{6}) \Delta(\zeta_{6}^{-2}q, \zeta_{6}, -1;q^2) &= -\frac{1}{2} (1-\zeta_3^{-1})\frac{J_2^6J_3^3J_{12}}{J_1^3J_4^3J_6^3}, \\
(1 - \zeta_{6}^2) \Delta(\zeta_{6}^{-4}q, \zeta_{6}^2, -1;q^2) &= -\frac{1}{2}(1+\zeta_3) \frac{J_1J_2^4J_{12}}{J_3J_4^3J_6}, \\
(1 - \zeta_{6}^3) \Delta(\zeta_{6}^{-6}q, \zeta_{6}^3, -1;q^2) &= 0 , \\
(1 - \zeta_{6}^4) \Delta(\zeta_{6}^{-8}q, \zeta_{6}^4, -1;q^2) &= -\frac{1}{2}(1+\zeta_3^{-1}) \frac{J_1J_2^4J_{12}}{J_3J_4^3J_6}, \\
(1 - \zeta_{6}^5) \Delta(\zeta_{6}^{-10}q, \zeta_{6}^5, -1;q^2) &= -\frac{1}{2}(1-\zeta_3) \frac{J_1J_2^4J_{12}}{J_3J_4^3J_6}.
\end{align*}
Together this gives
\begin{equation*}
\overline{D}(0,6) + \overline{D}(1,6) = \frac{1}{6}\left(3\frac{J_2^6J_3^3J_{12}}{J_1^3J_4^3J_6^3} + \frac{J_1J_2^4J_{12}}{J_3J_4^3J_6} \right).
\end{equation*}
To finish the proof of \eqref{6exampleeq1} we need to show that 
\begin{equation*}
3\frac{J_2^6J_3^3J_{12}}{J_1^3J_4^3J_6^3} + \frac{J_1J_2^4J_{12}}{J_3J_4^3J_6} = 4\frac{J_2^4}{J_1^2J_6}.
\end{equation*}
This follows as in the proof of \eqref{toprove2}. We omit the details.
\end{proof}

As a first corollary of Proposition \ref{6example}, we illustrate Remark \ref{remarkMeven} and find all of the rank deviations modulo $6$.   Let $T(q)$ denote the series on the right-hand side of \eqref{Tdef}.
\begin{corollary} \label{firstcor}
We have
\begin{align}
\overline{D}(0,6) &= -\frac{4}{3}q^2h(q^6;q^9) + \frac{1}{3}T(q) + \frac{4}{9}\frac{J_2^4}{J_1^2J_6} + \frac{2}{9}\frac{J_2J_3^6J_{18}}{J_1^2J_6^3J_9^2} , \label{6coreq1} \\
\overline{D}(1,6) &= \frac{4}{3}q^2h(q^6;q^9) - \frac{1}{3}T(q) + \frac{2}{9}\frac{J_2^4}{J_1^2J_6} - \frac{2}{9}\frac{J_2J_3^6J_{18}}{J_1^2J_6^3J_9^2}, \label{6coreq2} \\
\overline{D}(2,6) &= \frac{2}{3}q^2h(q^6;q^9) + \frac{1}{3}T(q) - \frac{2}{9}\frac{J_2^4}{J_1^2J_6} - \frac{1}{9}\frac{J_2J_3^6J_{18}}{J_1^2J_6^3J_9^2}, \label{6coreq3} \\
\overline{D}(3,6) &= -\frac{8}{3}q^2h(q^6;q^9) - \frac{1}{3}T(q) - \frac{4}{9}\frac{J_2^4}{J_1^2J_6} + \frac{4}{9}\frac{J_2J_3^6J_{18}}{J_1^2J_6^3J_9^2}. \label{6coreq4}
\end{align}
\end{corollary}

\begin{proof}
By \eqref{Tdef}, \eqref{6exampleeq1} and \eqref{6exampleeq2}, we have 
\begin{align}
\overline{D}(0,6) + 2\overline{D}(2,6) &= \overline{D}(0,6) + \overline{D}(2,6) + \overline{D}(4,6) \nonumber\\
&= \overline{D}(0,2) \nonumber \\
&= T(q) \label{add1}
\end{align}
and
\begin{equation} \label{add2}
2\overline{D}(0,6) - 2\overline{D}(2,6) = -4q^2h(q^6;q^9)  + \frac{4}{3}\frac{J_2^4}{J_1^2J_6} + \frac{2}{3}\frac{J_2J_3^6J_{18}}{J_1^2J_6^3J_9^2}.
\end{equation}
Adding \eqref{add1} and \eqref{add2} and multiplying by $1/3$ gives the formula for $\overline{D}(0,6)$ in \eqref{6coreq1}.    Equations \eqref{6coreq2}--\eqref{6coreq4} then follow from \eqref{6coreq1} and Proposition \ref{6example}.    
\end{proof}

For our second corollary, we show how Proposition \ref{6example} can be used to recover the following rank difference modulo $6$ \cite[Theorem 1.1]{jzz}.
\begin{corollary} \label{6corollary}
We have
\begin{equation} \label{jzzeq}
\begin{aligned}
\sum_{n \geq 0} &\left(\overline{N}(0,6,n) + \overline{N}(1,6,n) - \overline{N}(2,6,n) - \overline{N}(3,6,n) \right) q^n \\
&= \frac{J_{18}j(q^9;q^{18})}{J_6j(q^3;q^{18})^2} +2q\frac{J_{18}^3}{J_6j(q^3;q^{18})} +4q^2\frac{J_{18}^3}{J_6j(q^9;q^{18})} - 2q^2h(-q^3;q^9).
\end{aligned}
\end{equation}
\end{corollary}

\begin{proof}
Using \eqref{dic} and \eqref{hx+h-x}
we find that \eqref{jzzeq} is equivalent to
\begin{equation} \label{equivalentto}
\begin{aligned}
\overline{D}(0,6) &+ \overline{D}(1,6) - \overline{D}(2,6) - \overline{D}(3,6)  \\
&= \frac{J_6J_9^4}{J_3^2J_{18}^2} + 2q\frac{J_9J_{18}}{J_3} + 2q^2h(q^3;q^9).
\end{aligned}
\end{equation}
We also need the easily verified 
\begin{equation} \label{ez}
\frac{J_2^4}{J_1^2J_6} = \frac{J_6J_9^4}{J_3^2J_{18}^2} + 2q \frac{J_9J_{18}}{J_3} + q^2 \frac{J_{18}^4}{J_6J_9^2}.
\end{equation}
Now using (\ref{ez}) along with \eqref{3dissection}, \eqref{6exampleeq1} and \eqref{6exampleeq3} gives \eqref{equivalentto}, and the proof of (\ref{jzzeq}) is complete.
\end{proof}

\section*{Acknowledgements}
The authors would like to thank the Mathematisches Forschungsinstitut Oberwolfach for their support as this work began during their stay from January 9-22, 2022 as part of the Research in Pairs program. The second author was partially funded by a SSHN 2022 grant from the Embassy of France in Ireland during his visit to the Universit\'e Paris Cit\'e from December 4-17, 2022 and by the Irish Research Council Advanced Laureate Award IRCLA/2023/1934. Finally, the authors are grateful to the Max-Planck-Institut f{\"u}r Mathematik for their hospitality and support during their joint stay from May 1-31, 2023.

\section*{Data availability}
Data sharing is not applicable to this article as no datasets were generated or analysed during the current study.


\begin{thebibliography}{999}

\bibitem{bfhy}
M. Bian, H. Fang, X.Q. Huang and O.X.M. Yao, \emph{Ranks, cranks for overpartitions and Appell--Lerch sums}, Ramanujan J. \textbf{57} (2022), no. 2, 823--844.

\bibitem{bl}
K. Bringmann, J. Lovejoy, \emph{Dyson's rank, overpartitions, and weak Maass forms}, Int. Math. Res. Not. IMRN 2007, no. 19, Art. ID rnm063, 34 pp.

\bibitem{Br-Lo1}
K. Bringmann, J. Lovejoy, \emph{Overpartitions and class numbers of binary quadratic forms}, Proc. Natl. Acad. Sci. USA \textbf{106} (2009), no. 14, 5513--5516.

\bibitem{bo}
K. Bringmann, K. Ono, \emph{Dyson's ranks and Maass forms}, Ann. of Math. (2) \textbf{171} (2010), no. 1, 419--449.

\bibitem{bor}
K. Bringmann, K. Ono, R.C. Rhoades, \emph{Eulerian series as modular forms}, J. Amer. Math. Soc. \textbf{21} (2008), no. 4, 1085--1104.

\bibitem{Ci1}
A. Ciolan, \emph{Ranks of overpartitions: asymptotics and inequalities}, J. Math. Anal. Appl. \textbf{480} (2019), no. 2, 123444, 28 pp.

\bibitem{Ci2}
A. Ciolan, \emph{Inequalities between overpartition ranks for all moduli}, Ramanujan J. {\bf 58} (2022), no. 2, 463--489

\bibitem{cl}
S. Corteel, J. Lovejoy, \emph{Overpartitions}, Trans. Amer. Math. Soc. \textbf{356} (2004), no. 4, 1623--1635.

\bibitem{cgs}
S.-P. Cui, N.S.S. Gu and C.-Y. Su, \emph{Ranks of overpartitions modulo $4$ and $8$}, Int. J. Number Theory \textbf{16} (2020), no. 10, 2293--2310.

\bibitem{Dewar}
M. Dewar, \emph{The nonholomorphic parts of certain weak Maass forms}, J. Number Theory \textbf{130} (2010), no. 3, 559--573.

\bibitem{FJM}
J.-F. Fortin, P. Jacob and P. Mathieu, \emph{Jagged partitions}, Ramanujan J. \textbf{10} (2005), no. 2, 215--235.

\bibitem{FG}
J. Frye, F. Garvan, \emph{Automatic proof of theta-function identities} in: Elliptic integrals, elliptic functions and modular forms in quantum field theory, 195--258, Texts Monogr. Symbol. Comput., Springer, Cham, 2019.

\bibitem{gs}
N.S.S. Gu, C.-Y. Su, \emph{$M_2$-ranks of overpartitions modulo $4$ and $8$}, Ramanujan J. \textbf{55} (2021), no. 1, 369--392.

\bibitem{hm1}
D. R. Hickerson, E. T. Mortenson, \emph{Hecke-type double sums, Appell--Lerch sums, and mock theta functions, I}, Proc. Lond. Math. Soc. (3) \textbf{109} (2014), no. 2, 382--422.

\bibitem{hm2}
D. Hickerson, E. Mortenson, \emph{Dyson's ranks and Appell--Lerch sums}, Math. Ann. \textbf{367} (2017), no. 1-2, 373--395.

\bibitem{cjs}
C. Jennings-Shaffer, \emph{Overpartition rank differences modulo $7$ by Maass forms}, J. Number Theory \textbf{163} (2016), 331--358.

\bibitem{cjs2}
C. Jennings-Shaffer, \emph{The generating function of the $M_2$-rank of partitions without repeated odd parts as a mock modular form}, Trans. Amer. Math. Soc. \textbf{371} (2019), no. 1, 249--277.

\bibitem{jzz}
K. Ji, H. Zhang and A. Zhao, \emph{Ranks of overpartitions modulo $6$ and $10$}, J. Number Theory \textbf{184} (2018), 235--269.

\bibitem{Lo1}
J. Lovejoy, \emph{Rank and conjugation for the Frobenius representation of an overpartition}, Ann. Comb. \textbf{9} (2005), no. 3, 321--334. 

\bibitem{Lo2}
J. Lovejoy, \emph{Rank and conjugation for a second Frobenius representation of an overpartition}, Ann. Comb. \textbf{12} (2008), no. 1, 101--113.

\bibitem{Lo-Os1}
J. Lovejoy, R. Osburn, \emph{Rank differences for overpartitions}, Q. J. Math. \textbf{59} (2008), no. 2, 257--273.

\bibitem{Lo-Os2}
J. Lovejoy, R. Osburn, \emph{$M_2$-rank differences for overpartitions}, Acta. Arith. \textbf{144} (2010), no. 2, 193--212.

\bibitem{rm}
R. Mao, \emph{Asymptotics for rank moments of overpartitions}, Int. J. Number Theory \textbf{10} (2014), no. 8, 2011--2036.

\bibitem{rm2}
R. Mao, \emph{$M_2$-rank of overpartitions and harmonic weak Maass forms}, J. Math. Anal. Appl. \textbf{426} (2015), no. 2, 794--804.

\bibitem{Mc1}
R.J. McIntosh, \emph{The $H$ and $K$ family of mock theta functions}, Canad. J. Math. \textbf{64} (2012), no. 4, 935--960.

\bibitem{OnoCBMS}
K. Ono, \emph{The web of modularity: arithmetic of the coefficients of modular forms and $q$-series}, CBMS Regional Conference Series in Mathematics, 102, Published for the Conference Board of the Mathematical Sciences, Washington, DC by the American Mathematical Society, Providence, RI, 2004

\bibitem{wz}
B. Wei, H. Zhang, \emph{Rank differences for overpartitions modulo $6$}, Proc. Amer. Math. Soc. \textbf{148} (2020), no. 10, 4333--4349.

\bibitem{hz610}
H. Zhang, \emph{$M_2$-ranks of overpartitions modulo $6$ and $10$}, Ramanujan J. \textbf{51} (2020), no. 2, 353--389.

\bibitem{hz}
H. Zhang, \emph{Dyson's rank, overpartitions, and universal mock theta functions}, Canad. Math. Bull. \textbf{64} (2021), no. 3, 687--696.

\end{thebibliography}
\end{document}